\documentclass[11pt]{amsart}
\usepackage[T1]{fontenc}
\usepackage{lmodern,microtype}
\usepackage[margin=1in]{geometry}
\usepackage{amsmath,amssymb,amsthm,mathtools}

\usepackage{enumitem}
\usepackage[hidelinks]{hyperref}
\setlist[enumerate]{label=\textup{(\roman*)},leftmargin=2em}
\newtheorem{theorem}{Theorem}[section]
\newtheorem{proposition}[theorem]{Proposition}
\newtheorem{lemma}[theorem]{Lemma}
\newtheorem{corollary}[theorem]{Corollary}
\newtheorem{conjecture}[theorem]{Conjecture}
\theoremstyle{definition}

\newtheorem{remark}[theorem]{Remark}
\newcommand{\Aut}{\operatorname{Aut}}
\newcommand{\Sub}{\operatorname{Sub}}
\newcommand{\IRS}{\operatorname{IRS}}
\newcommand{\Ind}{\operatorname{Ind}}
\newcommand{\Cay}{\operatorname{Cay}}
\newcommand{\vol}{\operatorname{vol}}

\newcommand{\Sym}{\operatorname{Sym}}

\newcommand{\SL}{\operatorname{SL}}
\newcommand{\R}{\mathbb R}
\newcommand{\Q}{\mathbb Q}
\newcommand{\HH}{\mathbb H}
\newcommand{\dd}{\,d}
\title[The continuous Aldous--Lyons conjecture]{Counterexamples to the continuous\\ Aldous--Lyons conjecture}
\author{Tsachik Gelander}
\date{October 1, 2026}

\begin{document}
\begin{abstract}
{We construct non-co-sofic invariant random subgroups in
connected simple real Lie groups of rank one with finite center,
in Zariski connected simple rank-one groups over non-Archimedean
local fields, and in $\Aut(T_{a,b})$ for $a,b\geq3$.} This gives a negative answer to
the continuous analogue of the classical Aldous--Lyons conjecture
for these groups.
{Our constructions build on the existence of a finitely
generated non-sofic group and answer a question raised in the
author's ICM paper \cite[Section~6]{Gel18}.}
\end{abstract}
\maketitle

\section{Introduction}

An invariant random subgroup, or IRS, of a locally compact second
countable group $G$ is a conjugation-invariant Borel probability measure
on the Chabauty space $\Sub(G)$ of closed subgroups. A lattice
$\Lambda<G$ determines an IRS $\mu_\Lambda$ by conjugating $\Lambda$
with a random element of $G/\Lambda$, with respect to normalized
invariant measure. An IRS is \emph{co-sofic} if it can be approximated
by IRSs supported on lattices; equivalently, it belongs to the closed
convex hull of the lattice IRSs.

For bounded-degree random rooted graphs, the classical
Aldous--Lyons conjecture asked whether unimodularity suffices for
approximation by finite graphs \cite[Question~10.1]{AL07}.
It can equivalently be formulated as the assertion that every IRS
of the free group $F_n$ is co-sofic; see
\cite[Section~6]{Gel18} and \cite[Section~3]{AGN17}.
By the continuous Aldous--Lyons conjecture, we mean the co-soficity
question for invariant random subgroups of locally compact groups.
A negative answer to the classical Aldous--Lyons conjecture has been
given in \cite{BCLV24}.

For a free group $F_r$ and a normal subgroup $N\triangleleft F_r$,
the IRS $\delta_N$ is co-sofic if and only if the group $F_r/N$ is sofic.
This observation motivates the questions in \cite[Section~6]{Gel18}
about co-soficity in $\SL_2(\R)$, $\SL_2(\Q_p)$ and automorphism
groups of regular trees, since each of these groups admits a free
group $F_r$ as a lattice.
{For simple real Lie groups, the corresponding approximation
question was posed in \cite[Question~5]{ABBGNRS20}.}
\begingroup
For connected center-free simple real Lie groups of real rank at least
two, every discrete IRS is co-sofic. This follows from the
Nevo--St\"uck--Zimmer classification of ergodic IRSs
\cite[Theorem~1.14]{ABBGNRS17}, residual finiteness of lattices,
and ergodic decomposition. Levit extended the Stuck--Zimmer theorem
to non-Archimedean local fields, including positive characteristic
\cite[Theorem~1.2 and Corollary~1.3]{Lev17}. Thus, discrete IRSs
are co-sofic also in higher-rank non-Archimedean groups; see
\cite[Section~8]{GL18}.

Our main result is the following.

\begin{theorem}\label{thm:main}
Every connected simple real Lie group of real rank one with finite
center admits a non-co-sofic invariant random subgroup.
\end{theorem}

We also prove the analogous result for automorphism groups of regular
and biregular trees and for rank-one groups over non-Archimedean local
fields, in arbitrary characteristic.

\begin{theorem}\label{thm:trees-padic}
Let $k$ be a non-Archimedean local field of arbitrary characteristic,
and let $\mathbf G$ be a connected simple $k$-group of $k$-rank one.
Then $\mathbf G(k)$ admits a non-co-sofic invariant random subgroup.
The same holds for $\Aut(T_{a,b})$, where $T_{a,b}$ is the biregular
tree with vertex degrees $a,b\geq3$.
\end{theorem}

Starting with a uniform lattice $\Gamma<G$ admitting a finitely
generated non-sofic quotient $Q=\Gamma/N$, the strategy is to show
that the IRS induced from $\delta_N$ is non-co-sofic. This induced
IRS is given by
{
\begin{equation}\label{eq:induced}
 \mu_N=\Ind_\Gamma^G\delta_N
       =\int_{G/\Gamma}\delta_{gNg^{-1}}\dd m(g\Gamma).
\end{equation}}
Here $m$ is the normalized invariant measure on $G/\Gamma$.

Non-sofic groups were recently constructed in
\cite{FF26,KT26,OpenAI26}. The existence of non-sofic groups supplies
non-co-sofic \emph{Dirac} IRSs in the free group.

{In rank one, cocompact lattices $\Gamma$ are non-elementary
hyperbolic. Ol'shanskii's theorem
\cite[Theorem~1]{Ols95} embeds any given countable (non-sofic) group
in some quotient $Q=\Gamma/N$ of such a lattice. Since soficity passes
to subgroups, $Q$ is non-sofic. The difficulty lies in showing that
$\Ind_\Gamma^G\delta_N$ is non-co-sofic in $G$.
We conjecture that, in general, the induction of a non-co-sofic IRS
from a uniform lattice remains non-co-sofic
(see Conjecture~\ref{conj:induction}), but we prove this only in the
special cases of Theorems~\ref{thm:main} and~\ref{thm:trees-padic}.
Our proofs make use of the analytic and algebraic structure of these
groups.\par}

\endgroup

\subsection*{Sketch of proof}
\textbf{Lie groups.}
The key step is the transfer proposition (Proposition~\ref{thm:lie-transfer}).
Let $N\triangleleft\Gamma$, where $\Gamma<G$ is a uniform lattice,
and suppose that $C_G(N)$ is discrete. When $G$ is simple,
discreteness of the centralizer follows from Borel density and
normality of $N$.
If the induced IRS were
co-sofic, there would be lattices $\Lambda_i<G$ whose random
conjugates converge in distribution to a random conjugate of $N$.
We turn these continuous approximations into finite models
of the quotient $Q=\Gamma/N$.

We first choose a local transversal in $\Sub(G)$ near $N$:
a family of subgroups to which every sufficiently nearby subgroup
can be moved by a unique conjugation close to the identity.
Discreteness of the centralizer makes this possible by the inverse
function theorem. In $G/\Lambda_i$, we select the cosets
$x=g\Lambda_i$ whose stabilizers $g\Lambda_i g^{-1}$ belong to a
small neighborhood of $N$ within this transversal. The selected
cosets form a finite set. Indeed, translating each of them by the
same small identity neighborhood in $G$ gives disjoint sets of equal
positive volume inside the finite-volume space $G/\Lambda_i$.
The convergence of the IRSs then implies that, among these selected
cosets, the proportion whose stabilizers lie in any fixed smaller
neighborhood of $N$ tends to one.

To model an element $q\in Q$, choose a representative $\gamma\in\Gamma$.
Starting from a selected coset $x$, move to $\gamma x$ and then move
it slightly so that its stabilizer returns to the transversal.
For stabilizers sufficiently close to $N$, uniqueness of this
correction makes successive moves obey multiplication in $\Gamma$;
elements of $N$ fix the point, while different cosets of $N$ send it
to different points. For any fixed finite subset of $Q$, these rules
hold, and the required moves stay in our finite set, on a proportion
of points tending to one. Extending the partially defined moves to
permutations therefore gives sofic approximations of $Q$, a
contradiction when $Q$ is non-sofic.

\textbf{Trees.}
Here a graph construction replaces the local transversal. Starting
with a finitely generated non-sofic group $Q$, we construct a regular
or biregular graph $Y$ on which $Q$ acts freely with finitely many
vertex orbits. Its local unlabelled structure encodes the labelled
Cayley graph of $Q$. The covering $Y\to Q\backslash Y$ realizes
$Y=N\backslash T$ with $N\triangleleft\Gamma$, where $\Gamma$ is a
free uniform tree lattice and $\Gamma/N\cong Q$. Co-soficity of the
induced IRS would yield finite graph approximations to $Y$. To obtain
these from arbitrary approximating lattices, we discard vertices
with nontrivial local stabilizer action and use finite covolume to
extract finite components with uniform vertex weights. Local decoding
then gives finite approximations to the Cayley graph of $Q$, again
contradicting non-soficity.

\textbf{Non-Archimedean groups.}
A rank-one group over a non-Archimedean local field acts on its
Bruhat--Tits tree with compact vertex stabilizers and a type-preserving
subgroup transitive on edges. These properties let us realize the
finite quotient graph of the preceding construction by a free uniform
lattice in the group. Passing to its image in the tree automorphism
group and inducing there would preserve co-soficity, so the tree
argument gives the same contradiction. This works in both zero and
positive characteristic.

\medskip

{After the preliminaries in Section~\ref{sec:prelim},
Section~\ref{sec:lie} proves the Lie-group transfer theorem and
its rank-one applications. Sections~\ref{sec:graphs}--\ref{sec:padic}
then treat the regular and biregular graph encodings, tree lattices,
and rank-one groups over non-Archimedean local fields.}

\section{Approximation and induction}\label{sec:prelim}

All locally compact groups considered below are second countable.
Weak convergence of IRSs refers to the compact Chabauty space.
All graph degrees count edge multiplicity. The finite graphs used in
approximations may be disconnected.

\subsection{Finite labelled approximations}
Fix a finite generating set $S=\{s_1,\ldots,s_r\}$ of a group $Q$,
with $1_Q\notin S$.
The right Cayley graph $\Cay(Q,S)$ is the directed labelled
graph with vertex set $Q$ and an edge labelled $i$ from $q$ to
$qs_i$ for each $q\in Q$ and $1\leq i\leq r$. Similarly, permutations
$\sigma_1,\ldots,\sigma_r\in\Sym(V)$ of a finite nonempty set $V$
define a directed labelled multigraph $X$ by placing an edge labelled
$i$ from $v$ to $\sigma_i(v)$.

For such a graph $X$, let $B_k(X,v)$ denote the rooted subgraph
induced by the vertices at distance at most $k$ from $v$, with distance
measured after forgetting edge directions. We retain all edge
directions and labels, including loops and multiple edges.
Isomorphisms of these balls are required to preserve the root,
directions and labels.

The group $Q$ is \emph{sofic} if there exist finite nonempty sets
$V_n$ and permutations $\sigma_{1,n},\ldots,\sigma_{r,n}\in\Sym(V_n)$
such that the associated graphs $X_n$ satisfy, for every integer
$k\geq1$,
\[
 \frac{\bigl|\{v\in V_n:
 B_k(X_n,v)\cong B_k(\Cay(Q,S),1_Q)\}\bigr|}{|V_n|}
 \longrightarrow 1\qquad\text{as }n\to\infty.
\]
Equivalently, for every $k\geq1$ and every $\varepsilon>0$, there is a finite
permutation graph $X$ in which at least a proportion $1-\varepsilon$
of the vertices have the required radius-$k$ ball. This condition is
independent of the chosen finite generating set.

We will repeatedly use the following elementary observation.

\begin{lemma}[Completion of partial permutations]\label{lem:completion}
Let $V_n$ be finite nonempty sets. For $1\leq i\leq r$, let
$\tau_{i,n}$ be a bijection between two subsets of $V_n$, and let
$D_n$ be the directed labelled graph with an edge labelled $i$ from
$v$ to $\tau_{i,n}(v)$ whenever the latter is defined. Suppose that,
for every $k\geq1$, all but $o(|V_n|)$ vertices $v$ satisfy
$B_k(D_n,v)\cong B_k(\Cay(Q,S),1_Q)$. Then $Q$ is sofic.
\end{lemma}

\begin{proof}
The complements of the domain and range of $\tau_{i,n}$ have equal
cardinality. Choose an arbitrary bijection between them, obtaining
$\sigma_{i,n}\in\Sym(V_n)$. Fix $k\geq1$. If the radius-$(k+1)$
ball of $D_n$ at $v$ agrees with the Cayley graph, then every vertex
within distance $k$ of $v$ already has all its incoming and outgoing
labelled edges. Completion therefore leaves the radius-$k$ ball at
$v$ unchanged. The hypothesis at radius $k+1$ now gives the required
convergence at radius $k$.
\end{proof}

Let $Y$ be a connected, locally finite, unlabelled, undirected
multigraph on which $Q$ acts freely with finitely many orbits.
Choose vertices $y_1,\ldots,y_m$,
one from each vertex orbit, and choose the root uniformly from this
finite set. The resulting distribution on rooted graphs is
\[
 \nu_Y=\frac1m\sum_{i=1}^m\delta_{(Y,y_i)},
\]
where $(Y,y_i)$ denotes the rooted isomorphism class of $Y$ with root
$y_i$. This distribution does not depend on the chosen representatives,
since vertices in the same orbit are related by a graph automorphism.

We call $\nu_Y$ \emph{graph-sofic} if there is a sequence of finite
unlabelled, undirected multigraphs $Y_n$ of uniformly bounded degree with the following
property: for every fixed radius $k$, the distribution of the rooted
radius-$k$ ball around a uniformly chosen vertex of $Y_n$ converges
to the distribution of the radius-$k$ ball under $\nu_Y$. This is
convergence in the local weak topology.

\subsection{Lattice approximation}
The closed-convex-hull formulation of co-soficity agrees with the
definition using IRSs supported on lattices: such measures decompose
into lattice orbit measures, and their barycentres are approximated
by finite convex combinations. We will use the following familiar
extremality observation.

\begin{lemma}\label{lem:ergodic}
If an ergodic IRS $\mu$ is co-sofic, there is a sequence of lattices
$\Lambda_n<G$ such that $\mu_{\Lambda_n}\to\mu$.
\end{lemma}

\begin{proof}
Let $\mathcal L$ be the set of lattice IRSs and put
$C=\overline{\operatorname{conv}}(\mathcal L)\subset\IRS(G)$.
This is a compact convex set in the space of signed measures on
$\Sub(G)$ with the weak-* topology. Co-soficity gives $\mu\in C$,
and ergodicity makes $\mu$ extreme in $\IRS(G)$, hence also in $C$.
Milman's converse to the Krein--Milman theorem
\cite[Proposition~1.5]{Phelps01} gives
$\operatorname{ext}(C)\subset\overline{\mathcal L}$.
Thus $\mu\in\overline{\mathcal L}$, and metrizability yields the
required sequence.
\end{proof}

\subsection{\texorpdfstring{{Induction and compact quotients}}{Induction and compact quotients}}
If a closed subgroup $H<G$ has finite invariant covolume, induction
of an IRS $\eta$ of $H$ is defined by averaging $g_*\eta$ over
$G/H$, where $g_*$ denotes conjugation. Invariance of $\eta$ makes
this independent of the chosen section.

\begin{lemma}\label{lem:induction}
Suppose that $H<G$ is closed and cocompact and that $G/H$ has a
$G$-invariant probability measure. Then induction
$\IRS(H)\to\IRS(G)$ is continuous and preserves co-soficity.
Induction is transitive along chains of finite-covolume subgroups.
{If $\pi:\widetilde G\to G$ is a quotient homomorphism with compact
kernel, the pushforward $L\mapsto\pi(L)$ on IRSs is continuous
and preserves co-soficity.}
\end{lemma}

\begin{proof}
For a continuous function $f$ on $\Sub(G)$, the integral defining
induction is continuous in $\eta$ by dominated convergence, after
choosing a measurable section with relatively compact image.
Such a section exists because $G/H$ is compact. A lattice in $H$
is a lattice in $G$, and induction of its lattice IRS is its
$G$-lattice IRS. This proves preservation of co-soficity.
Transitivity follows by disintegration of invariant measures.

A quotient homomorphism with compact kernel induces a continuous
map $\Sub(\widetilde G)\to\Sub(G)$, $L\mapsto\pi(L)$: compactness
of the kernel gives the necessary compact lifts in the Chabauty
criterion. {Properness of $\pi$ implies that the image of a
discrete subgroup is discrete. Pushing forward the finite invariant measure
on a lattice quotient shows that images of lattices are lattices.
For a finite kernel, full inverse images of lattices are also lattices.}
The asserted preservation follows.
\end{proof}

If $N\triangleleft\Gamma$ and $\Gamma$ is cocompact, the map
$G/\Gamma\to\Sub(G)$ in \eqref{eq:induced} is continuous.
Its image is exactly the conjugacy orbit of $N$ and is compact.
The resulting IRS is ergodic, as a factor of the transitive
probability $G$-space $G/\Gamma$.

\begingroup
\section{Rank-one Lie groups}\label{sec:lie}

We first establish a transfer theorem for real Lie groups.

\subsection{\texorpdfstring{{A transfer theorem}}{A transfer theorem}}
{The next theorem justifies, under certain conditions, inducing
the Dirac measure on the kernel of a non-sofic quotient.\par}
\endgroup

\begingroup
\begin{proposition}[Transfer from a normal subgroup of a uniform lattice]
\label{thm:lie-transfer}
Let $G$ be a second countable real Lie group, let $\Gamma<G$ be a uniform
lattice, and let $N\triangleleft\Gamma$. Suppose that $C_G(N)$ is discrete
and that $Q=\Gamma/N$ is finitely generated. If
$\mu=\operatorname{Ind}_{\Gamma}^{G}\delta_N$ is co-sofic, then $Q$ is sofic.
\end{proposition}
\endgroup

\begin{proof}
By ergodicity of $\mu$ and Lemma~\ref{lem:ergodic}, there are lattices
$\Lambda_i<G$ such that $\nu_i:=\mu_{\Lambda_i}\to\mu$.
Let $X_i=G/\Lambda_i$, equipped with its invariant probability measure
$\lambda_i$. We use the left action of $G$ on $X_i$ and write
\[
 H_x=\operatorname{Stab}_G(x),\qquad
 H_{g\Lambda_i}=g\Lambda_i g^{-1}.
\]
Thus the distribution of $H_x$, for $x$ chosen according to $\lambda_i$,
is $\nu_i$. The lattices $\Lambda_i$ are not necessarily uniform or
torsion-free. Fix a Haar measure $m$ on $G$ and write
$v_i=\vol(G/\Lambda_i)$ for the corresponding covolume. Recall that $G$
is unimodular, since it admits a lattice.

We will extract finite sets from the spaces $X_i$ and construct almost
actions of $Q$ on them. More precisely, we use the following equivalent
formulation of soficity: for every finite $F\subset Q$ containing $1$ and
every $\varepsilon>0$, there must be a finite nonempty set $Z$ and
permutations $\sigma(q)$, $q\in F\cup F^2$, with
$\sigma(1)=\mathrm{id}$, such that at least $(1-\varepsilon)|Z|$ points
$x\in Z$ satisfy, simultaneously,
\[
 \sigma(q)\sigma(r)x=\sigma(qr)x\quad(q,r\in F),
 \qquad
 \sigma(q)x\ne\sigma(r)x\quad(q,r\in F,\ q\ne r).
\]
The finite sets will consist of points whose stabilizers lie in a local
transversal to conjugation. We first construct this transversal and
explain how convergence of the IRSs controls the counting measure on
these finite sets.

\smallskip
\noindent\textbf{1. A local slice in the space of subgroups.}
Since $C_G(N)$ is discrete and the Lie algebra of $G$ is finite-dimensional,
we can choose $n_1,\ldots,n_k\in N$ such that
\[
 \bigcap_{j=1}^k\ker(\operatorname{Ad}(n_j)-1)=0.
\]
This means that the simultaneous conjugation orbit of
$\mathbf n=(n_1,\ldots,n_k)$ in $G^k$ has dimension $\dim G$.

We need to follow these elements as the subgroup varies. First, all
subgroups sufficiently close to $N$ avoid a fixed identity neighborhood
apart from $1$. To see this, choose $\rho>0$ so small that the exponential
map is injective on the closed ball of radius $2\rho$ in the Lie algebra,
and its image meets $N$ only at $1$. Nearby subgroups avoid the compact
annulus
\[
 \{\exp X:\rho\leq\|X\|\leq2\rho\}.
\]
If such a subgroup contained $\exp X$ with $0<\|X\|<\rho$, some positive
integral power of this element would belong to the annulus, a
contradiction. This proves the required uniform discreteness.

Consequently, for $H$ near $N$, there is a unique element $h_j(H)\in H$
in a sufficiently small fixed neighborhood of $n_j$. Existence follows
from Chabauty convergence, and uniqueness from uniform discreteness.
These elements depend continuously on $H$. Moreover, for $u$ close to
$1$ and $H$ close to $N$, uniqueness gives
\[
 h_j(uHu^{-1})=u h_j(H)u^{-1}.
\]
Put $f(H)=(h_1(H),\ldots,h_k(H))$.

Choose a smooth transversal $S\subset G^k$ through $\mathbf n$ to its
simultaneous conjugation orbit. The inverse function theorem gives a
unique expression $f(H)=u\cdot\mathbf s$ for nearby tuples, where $u$ is
close to $1$, $\mathbf s\in S$, and $u\cdot\mathbf s$ denotes simultaneous
conjugation. Define
\[
 \Sigma=\{H\text{ near }N:f(H)\in S\},
\]
with the subspace topology from $\Sub(G)$. This is a family of subgroups,
and it need not be a manifold. Equivariance of $f$ shows that
$K=u^{-1}Hu$ satisfies $f(K)=\mathbf s$, so $K\in\Sigma$.
Uniqueness of the tuple coordinates gives uniqueness of $u$ and $K$.
These coordinates depend continuously on $H$; after shrinking
neighborhoods, they give a homeomorphism
\[
 \Phi:U_0\times\Sigma\longrightarrow\mathcal B\subset\Sub(G),
 \qquad (u,K)\longmapsto uKu^{-1},
\]
onto an open neighborhood of $N$. Here $U_0$ is an identity neighborhood
in $G$. For $H=uKu^{-1}\in\mathcal B$, set
\[
 c(H)=u^{-1},\qquad c(H)Hc(H)^{-1}=K\in\Sigma.
\]
Thus $c(H)$ is the unique sufficiently small correction taking $H$ into
the transversal, and $c(K)=1$ for $K\in\Sigma$. Only conjugations close
to the identity are involved in this uniqueness assertion.

We also arrange that $\Sigma$ meets the orbit of $N$ only at $N$.
Indeed, let $L=N_G(N)$. For each $n\in N$, the continuous map
$l\mapsto lnl^{-1}$ sends $L^\circ$ into the discrete set $N$, so it is
constant. Hence $L^\circ$ centralizes $N$, and $L$ is discrete.
Since $L\supset\Gamma$, the space $G/L$ is compact. Its continuous
injective map into $\Sub(G)$ is therefore an embedding with image $N^G$.
Every point of this orbit sufficiently close to $N$ is thus $uNu^{-1}$
with $u\in U_0$. Uniqueness of the local product implies, after shrinking
$\Sigma$,
\[
 \Sigma\cap N^G=\{N\}.
\]

\smallskip
\noindent\textbf{2. Counting points on the transversal.}
Choose a relatively compact open identity neighborhood $U$ with
$\overline U\subset U_0$ and $m(\partial U)=0$. For an open neighborhood
$W\ni N$ in $\Sigma$ whose closure is compact and contained in $\Sigma$,
put
\[
 B(W)=\Phi(U\times W),\qquad
 Z_i(W)=\{x\in X_i:H_x\in W\}.
\]
Here $W$ consists of subgroups, whereas $Z_i(W)$ consists of points of
$X_i$. We do not assert that every $H_x$ belongs to $W$. Nor is $W$
restricted to conjugates of $N$: it can contain lattices near $N$.

The local product gives an exact decomposition
\[
 \{y\in X_i:H_y\in B(W)\}
 =\bigsqcup_{x\in Z_i(W)}Ux.
\]
To see that these sets cover the left-hand side, write
$H_y=uKu^{-1}$ with $u\in U$ and $K\in W$. Then $x=u^{-1}y$ satisfies
$H_x=K$, so $x\in Z_i(W)$ and $y\in Ux$. To see disjointness, suppose
$ux=vy$ with $x,y\in Z_i(W)$ and $u,v\in U$. Taking stabilizers gives
$\Phi(u,H_x)=\Phi(v,H_y)$. Product uniqueness implies $u=v$, and then
$x=y$. The same argument shows that $u\mapsto ux$ is injective on $U$.

Each set $Ux$ consequently has quotient Haar volume $m(U)$, or
$\lambda_i$-measure $m(U)/v_i$. Since $X_i$ has finite volume, $Z_i(W)$
is finite, and the decomposition gives
\[
 \nu_i(B(W))=\frac{m(U)}{v_i}\,|Z_i(W)|.
\]
This identity converts convergence of the IRSs into information about
ordinary, unweighted counts of points.

\smallskip
\noindent\textbf{3. Concentration near $N$.}
The measure $\mu$ is supported on $N^G$. For every $W$ as above,
\[
 B(W)\cap N^G=\Phi(U\times\{N\}).
\]
Thus all these sets have the same positive $\mu$-measure, say $b$.
They are also $\mu$-continuity sets. Indeed, their closures stay within
the larger product chart, and
\[
 \partial B(W)\cap N^G=\Phi(\partial U\times\{N\}).
\]
This has $\mu$-measure zero because $m(\partial U)=0$ and the orbit
measure is induced by Haar measure on $G/L$.

It follows that $\nu_i(B(W))\to b>0$, so $Z_i(W)$ is nonempty for all
sufficiently large $i$. More importantly, if $W'\subset W$ is any smaller
neighborhood of $N$ with compact closure in $W$, the counting identity
gives
\begin{equation}\label{eq:transversal-concentration}
 \frac{|Z_i(W')|}{|Z_i(W)|}
 =\frac{\nu_i(B(W'))}{\nu_i(B(W))}
 \longrightarrow1.
\end{equation}
In words, almost every point of the finite set $Z_i(W)$ has its stabilizer
in the smaller neighborhood $W'$. This will allow us to make all the
local constructions below simultaneously valid on most of $Z_i(W)$.

\smallskip
\noindent\textbf{4. Holonomy: returning to the transversal.}
Fix $\gamma\in\Gamma$. If $H_x\in\Sigma$ is sufficiently close to $N$,
then $\gamma H_x\gamma^{-1}$ belongs to $\mathcal B$, since $\gamma$
normalizes $N$. Define
\[
 b_\gamma(H)=c(\gamma H\gamma^{-1}),\qquad
 T_\gamma x=b_\gamma(H_x)\gamma x.
\]
The first operation moves $x$ by $\gamma$; the correction then returns
its stabilizer to $\Sigma$. Explicitly,
\[
 H_{T_\gamma x}
 =b_\gamma(H_x)\gamma H_x\gamma^{-1}b_\gamma(H_x)^{-1}\in\Sigma.
\]
This stabilizer depends continuously on $H_x$, and the corresponding
map of subgroups fixes $N$. Also $b_\gamma(H)\to1$ as $H\to N$.
All the neighborhoods used here depend on subgroups and on $\gamma$,
not on $i$.

These return maps satisfy the group law locally. To check this, write
\[
 T_\gamma T_\eta x
 =\bigl[b_\gamma(H_{T_\eta x})\,
          \gamma b_\eta(H_x)\gamma^{-1}\bigr]\gamma\eta x.
\]
The bracketed element tends to $1$ as $H_x\to N$, and carries the
stabilizer of $\gamma\eta x$ into $\Sigma$. By uniqueness of the small
correction it equals $b_{\gamma\eta}(H_x)$. Therefore
$T_\gamma T_\eta=T_{\gamma\eta}$ near $N$. In particular,
$T_1=\mathrm{id}$ and $T_{\gamma^{-1}}$ is a local inverse of $T_\gamma$.

Elements of $N$ act trivially under these return maps. For a fixed
$n\in N$, let $h_n(H)\in H$ be its unique nearby representative, obtained
as in Step~1. The small element $h_n(H)n^{-1}$ conjugates $nHn^{-1}$
back to $H\in\Sigma$. Uniqueness gives
$b_n(H)=h_n(H)n^{-1}$, and hence
\[
 T_nx=h_n(H_x)x=x.
\]
On the other hand, distinct cosets act differently near $N$. If
$\gamma N\ne\eta N$ and $T_\gamma x=T_\eta x$ occurred for stabilizers
$H_x$ arbitrarily close to $N$, we could find $H_j\to N$ with
\[
 \bigl(b_\eta(H_j)\eta\bigr)^{-1}b_\gamma(H_j)\gamma\in H_j.
\]
Passing to the limit in the Chabauty topology would give
$\eta^{-1}\gamma\in N$, a contradiction. Thus the return maps give a
locally defined free action of the quotient $Q$, in the sense that each
prescribed finite collection of identities and inequalities holds on a
common neighborhood of $N$.

\smallskip
\noindent\textbf{5. Finite permutations.}
Fix a finite $F\subset Q$ containing $1$, and put $E=F\cup F^2$.
For each $q\in E$, choose a lift $\gamma_q\in\Gamma$, with $\gamma_1=1$.
The preceding local identities imply
\[
 T_{\gamma_q}T_{\gamma_r}=T_{\gamma_{qr}}
 \quad(q,r\in F),
\]
because $\gamma_{qr}^{-1}\gamma_q\gamma_r\in N$ and the return map
associated to this element is the identity. They also imply
$T_{\gamma_q}x\ne T_{\gamma_r}x$ for $q\ne r$ when $H_x$ is sufficiently
close to $N$.

Choose $W$ small enough that the finitely many maps $T_{\gamma_q}$,
$q\in E$, are defined and injective over $W$, and that the required
local identities apply near $N$. Since their induced maps of stabilizers
are continuous and fix $N$, there is a smaller neighborhood $W'\ni N$,
with compact closure in $W$, such that, whenever $H_x\in W'$, all the
points
\[
 T_{\gamma_q}x\quad(q\in E),\qquad
 T_{\gamma_q}T_{\gamma_r}x\quad(q,r\in F)
\]
have stabilizers in $W$, and the preceding identities and inequalities
hold at $x$. These choices are independent of $i$.

For large $i$, let $Z=Z_i(W)$. Restrict each $T_{\gamma_q}$ to those
points of $Z$ whose images also lie in $Z$. This is a bijection between
two subsets of the finite set $Z$, so it extends to a permutation
$\sigma_i(q)$ of $Z$; take $\sigma_i(1)=\mathrm{id}$.
For $x\in Z_i(W')$, our choice of $W'$ ensures that every intermediate
point needed below remains in $Z$. The extensions therefore agree with
the return maps at these points, giving
\[
 \sigma_i(q)\sigma_i(r)x
 =T_{\gamma_q}T_{\gamma_r}x
 =T_{\gamma_{qr}}x=\sigma_i(qr)x
 \quad(q,r\in F).
\]
Likewise $\sigma_i(q)x\ne\sigma_i(r)x$ for distinct $q,r\in F$.
By~\eqref{eq:transversal-concentration}, the proportion of points in
$Z_i(W')$ tends to one. Choosing $i$ large enough proves the
finite-permutation criterion stated at the beginning, for the given
$F$ and any $\varepsilon>0$. Hence $Q$ is sofic.
\end{proof}

\begingroup
\subsection{The proof of Theorem~\ref{thm:main}}\label{subsec:proof-rank-one}

\begin{proof}
First suppose that $G$ has trivial center, and choose a torsion-free
cocompact lattice $\Gamma<G$. The group $\Gamma$ is non-elementary
word-hyperbolic. By Ol'shanskii's SQ-universality theorem
\cite[Theorem~1]{Ols95}, every countable group embeds in a quotient
of $\Gamma$. Choose a finitely generated non-sofic group $P$
\cite[Theorem~1.3]{FF26}; see also \cite{OpenAI26,KT26}. There is
therefore a normal subgroup $N\triangleleft\Gamma$ such that
$Q=\Gamma/N$ contains $P$. Since soficity passes to subgroups, $Q$
is non-sofic; it is finitely generated because $\Gamma$ is.

The lattice $\Gamma$ is linear and residually finite, hence sofic.
Thus $N\ne\{1\}$, and torsion-freeness makes $N$ infinite. By
Borel density and normality, the Zariski closure of $N$ is a
positive-dimensional normal algebraic subgroup of the simple
adjoint group. Hence $N$ is Zariski dense and $C_G(N)$ is discrete.
Proposition~\ref{thm:lie-transfer} shows that
$\Ind_\Gamma^G\delta_N$ is non-co-sofic. It is discrete and ergodic,
and its support is the compact image of $G/\Gamma$.

For general $G$ with finite center, apply this construction to
$G/Z(G)$ and take the full inverse images of the lattice and its
normal subgroup in $G$. The resulting induced IRS is again discrete
and ergodic, with compact conjugacy orbit. If it were co-sofic, its
image in $G/Z(G)$ would be co-sofic by
Lemma~\ref{lem:induction}, a contradiction.
\end{proof}
\endgroup

\begingroup
\begin{remark}\label{rem:prescribed-quotient}
For real hyperbolic groups, one can prescribe the non-sofic quotient
$Q$. For every $n\geq2$, there is a torsion-free cocompact lattice
in $\operatorname{SO}^{+}(n,1)$ mapping onto $F_2$;
see \cite{Lub96} and \cite[pp.~1161--1162]{BGLM02}.
Passing to a finite-index subgroup and composing with a free-group
surjection onto $Q$ gives a uniform lattice $\Gamma$ and
$N\triangleleft\Gamma$ with $\Gamma/N\cong Q$, for any finitely
generated non-sofic $Q$. As in the proof of Theorem~\ref{thm:main},
$N$ is Zariski dense, so Proposition~\ref{thm:lie-transfer} makes the
induced IRS non-co-sofic, also in the full group
$\operatorname{SO}(n,1)$. Taking finite central inverse images gives
the same realization property for $\SL_2(\R)$.
\end{remark}
\endgroup

\section{\texorpdfstring{{Regular and biregular graph encodings}}{Regular and biregular graph encodings}}\label{sec:graphs}

{We now turn to trees and non-Archimedean local fields.
We first give a regular graph encoding, then extend it to arbitrary
biregular trees. Both encodings recover soficity of a group from finite
graph approximations of an unlabelled tree quotient.\par}

\begin{proposition}\label{prop:encoding}
Let $Q=\langle s_1,\ldots,s_r\rangle$ with $r\geq2$, and let
$d\geq3$. There is a connected bipartite $d$-regular multigraph
$Y_Q$ with a free cocompact action of $Q$ such that
\[
 \nu_{Y_Q}\text{ graph-sofic}\quad\Longrightarrow\quad Q\text{ sofic}.
\]
The quotient map $Y_Q\to B=Q\backslash Y_Q$ is a regular graph
covering with deck group $Q$.
\end{proposition}

\begin{proof}
Replace each vertex $q$ of the Cayley graph by a cycle with ports
\[
 1^+,1^-,\ldots,r^+,r^-
\]
in this cyclic order. Join port $i^+$ at $q$ to port $i^-$ at
$qs_i$ by a path of length $7$. In each cycle, replace the edge
$1^+1^-$ by a path of length $3$ and the edge $r^-1^+$ by a path
of length $5$; leave the other cycle edges unsubdivided.
All path interiors are disjoint.

To make the graph $d$-regular, give each odd-length path alternating
edge multiplicities $a,d-a,a,\ldots,a$, taking $a=1$ for cycle paths
and $a=d-2$ for connecting paths. Interior vertices then have degree
$a+(d-a)=d$, and ports have degree $1+1+(d-2)=d$.
The graph is connected and bipartite: give the plus and minus ports
opposite colours and extend along the odd-length paths.
Left translation of the cycle indices gives the required free
$Q$-action with finite quotient, and the quotient map is a covering.

The construction can be decoded locally. Ports are exactly the
vertices with three distinct neighbours; path interiors have two.
Thus the paths between ports and their lengths are recognizable.
Length $7$ identifies the connecting paths. The remaining paths
form cycles, and the unique common port of the length-$3$ and
length-$5$ paths identifies $1^+$ and the cyclic orientation.
Hence every port label is recovered by a bounded-radius rule.

Suppose finite multigraphs $Z_n$ converge locally to $\nu_{Y_Q}$.
The cycles we retain in $Z_n$ are specified as follows, identifying
paths and cycles in the underlying simple graph. A retained cycle
has exactly $2r$ vertices with three distinct neighbours, called
its ports; its other vertices have exactly two distinct neighbours.
The paths along the cycle between successive ports have lengths
\[
 3,\underbrace{1,\ldots,1}_{2r-2\text{ times}},5
\]
in cyclic order. At each port, the edge leaving the cycle must
extend to a path of length $7$ ending at another vertex with three
distinct neighbours, with exactly two distinct neighbours at each
interior vertex. All vertices on the cycle and these additional
paths must have degree $d$, and the edge multiplicities must be
those prescribed above. These are bounded-radius conditions.
Retained cycles are disjoint: at each port the two cycle paths
are uniquely distinguished from the length-$7$ path by their lengths.
Their $1^+$ ports form a set $V_n$ of anchors. There is one anchor
orbit in $Y_Q$, so
\[
 \frac{|V_n|}{|V(Z_n)|}\longrightarrow
 \frac{1}{|V(Q\backslash Y_Q)|}>0.
\]
The length-$7$ paths between matching $i^+$ and $i^-$ ports of
retained cycles define partial bijections of $V_n$.
For each fixed $k$, a sufficiently large correct neighbourhood
at an anchor decodes to the labelled radius-$k$ Cayley ball.
Local convergence and the positive anchor density imply that
all but $o(|V_n|)$ anchors have this property.
Lemma~\ref{lem:completion} completes the partial bijections to
permutations and proves that $Q$ is sofic.
\end{proof}

\begin{remark}
The use of multigraphs is convenient and harmless. A connected
$d$-regular multigraph without loops has universal covering tree
$T_d$, and parallel edges are retained in both local convergence
and the covering construction. The graph above has no loops.
\end{remark}

\begingroup
\begin{proposition}\label{prop:biregular-encoding}
For every $a,b\geq3$ and every finitely generated group $Q$ with a generating tuple of size at least two, there is a connected simple bipartite $(a,b)$-biregular graph $Y_Q$, with a free type-preserving $Q$-action and finite quotient, such that
\[
 \nu_{Y_Q}\text{ graph-sofic}\quad\Longrightarrow\quad Q\text{ sofic}.
\]
The map $Y_Q\to Q\backslash Y_Q$ is a regular covering with deck group $Q$.
\end{proposition}

\begin{proof}
Let $S_Q$ be the graph constructed in Proposition~\ref{prop:encoding}
before edge multiplicities are assigned. Its ports have degree $3$
and its other vertices degree $2$. The distinguished cycles have
length $2r+6\geq10$, and every other cycle contains a length-$7$
connecting path and at least one further edge. Thus $S_Q$ is simple,
bipartite and has girth at least $8$. Its $Q$-action is free with
finitely many vertex orbits. The same port-and-path decoding as in
Proposition~\ref{prop:encoding}, omitting only the edge-multiplicity
tests, shows that graph-soficity of $\nu_{S_Q}$ implies soficity of $Q$.

\textbf{The blocks.}
For each vertex $v$ of $S_Q$ of degree $j\in\{2,3\}$, take a copy $C_v$ of $K_{b,a}$: its type-$0$ part has $b$ vertices and its type-$1$ part has $a$ vertices. Delete a matching of size $j$, and assign its deleted edges bijectively to the edges of $S_Q$ incident to $v$. Denote the two endpoints assigned to an incident edge $e$ by $u_{v,e}$ of type $0$ and $w_{v,e}$ of type $1$. For every edge $e=vv'$ of $S_Q$, add the two edges
\[
 u_{v,e}w_{v',e},\qquad u_{v',e}w_{v,e}.
\]
Make these choices $Q$-equivariantly, by first making them on orbit representatives. This is possible because the action on $S_Q$ is free on vertices and has no edge inversions.

Each deficient vertex receives exactly one new edge. Thus $Y_Q$ is $(a,b)$-biregular. It is connected and simple. The $Q$-action is free and type-preserving, with finitely many orbits. Its quotient map is therefore a regular covering.

\textbf{Recognizing the blocks.}
We claim that an edge of $Y_Q$ is internal to a block precisely when it lies on a simple cycle of length at most $6$.

First, every edge of $K_{b,a}$ minus a matching of size at most $3$ lies on a cycle of length $4$ or $6$. Indeed, for an edge $uw$, choose neighbours $u'\ne u$ of $w$ and $w'\ne w$ of $u$. If some such $u',w'$ are adjacent there is a $4$-cycle. Otherwise the matching condition forces both choice sets to be singletons. Then $a=b=3$, all three matching edges have been deleted, and the block is a $6$-cycle. These blocks are connected: two vertices on the type-$0$ side have at least $a-2\geq1$ common neighbours, and every type-$1$ vertex has at least $b-1\geq2$ neighbours.

Next, a cycle of length at most $6$ cannot use edges between blocks. If it did, the edges of $S_Q$ traversed by its projection would span a forest, since $S_Q$ has girth at least $8$. Choose a leaf of the nontrivial projected tree. The cycle enters and leaves that leaf block through the two edges associated with one deleted matching edge. Its path inside the leaf block has length at least $3$. In the rest of the cycle, the two endpoints in the neighbouring block are also the endpoints of a deleted matching edge. They are nonadjacent and have opposite types, so the complementary path between them has length at least $3$. Together with the two edges between blocks, this gives length at least $8$, a contradiction.

Consequently the connected components of the subgraph consisting of edges on cycles of length at most $6$ are exactly the blocks $C_v$. Contract each block and replace each resulting pair of parallel edges by one edge. This recovers $S_Q$.

\textbf{Finite approximations.}
This recovery is local: every block has exactly $a+b$ vertices and bounded diameter. In a finite approximation, keep only components of short-cycle edges that are copies of the prescribed blocks and whose vertices have the correct ambient degrees $a$ and $b$. Connect two retained blocks only when their external edges form the prescribed pair. Discard malformed incidences. The resulting finite graph has maximum degree $3$. Recognition of a complete block, including the absence of further short-cycle edges leaving it, requires only a bounded neighbourhood. Thus all but a vanishing proportion of vertices belong to correctly recognized blocks, and for every fixed radius almost every retained block has the correct decoded neighbourhood.

Every block has the same size $a+b$. Choosing a retained block uniformly therefore has the same asymptotic weighting as choosing an original vertex uniformly and taking its block. The decoded finite graphs converge to $\nu_{S_Q}$. The decoding of $S_Q$ described above now proves the implication.
\end{proof}

\endgroup

\section{Tree lattices and finite graph approximations}\label{sec:trees}

\begingroup
Let $a,b\geq3$, with vertices of types $0$ and $1$ having degrees $a$ and $b$, respectively. Write $A=\Aut(T_{a,b})$ and let $A^+$ be its type-preserving subgroup. The group $A$ is unimodular: its type-preserving subgroup is generated by compact vertex stabilizers and has index at most two. Fix adjacent vertices $o_0,o_1$ of the two types. For an IRS $\mu$ supported on type-preserving subgroups, define its quotient law $\mathcal L(\mu)$ by sampling $H\sim\mu$ and rooting $H\backslash T_{a,b}$ at the image of $o_t$, where
\begin{equation}\label{eq:tree-root}
 \Pr(t=0)=\frac{b}{a+b},\qquad
 \Pr(t=1)=\frac{a}{a+b}.
\end{equation}
We forget the type labels in the resulting rooted graph. When $a=b$, $A$ is vertex-transitive and this is also the law obtained from either fixed root.

\begin{lemma}\label{lem:tree-extraction}
If $\mu$ is a co-sofic IRS of $A$ supported on type-preserving subgroups acting freely on vertices, then $\mathcal L(\mu)$ is graph-sofic.
\end{lemma}

\begin{proof}
Approximate $\mu$ by finite convex combinations of lattice IRSs. When $a=b$, replace each approximating subgroup by its intersection with $A^+$. Intersection with an open subgroup is continuous on Chabauty spaces, and this operation sends a lattice IRS of $A$ to the IRS of the finite-index lattice $\Lambda\cap A^+$. It fixes $\mu$. We can therefore assume all approximating lattices preserve types.

\textbf{Root weights.}
For a lattice $\Lambda<A$, put $X=\Lambda\backslash T_{a,b}$ and
\[
 w(v)=|\Lambda_{\widetilde v}|^{-1}\quad(v\in V(X)).
\]
Here $\widetilde v$ is any lift. The covolume formula gives $\sum_vw(v)<\infty$; see \cite[Chapter~1]{BL01}. If $a\ne b$, normalize the Haar measure of an edge stabilizer to be one. The two vertex stabilizers then have measures $a$ and $b$. Writing $V=\operatorname{vol}(A/\Lambda)$, the fixed-type root probabilities are
\[
 \Pr(v\mid t=0)=\frac{a}{V}w(v),\qquad
 \Pr(v\mid t=1)=\frac{b}{V}w(v).
\]
After the mixing in \eqref{eq:tree-root}, every vertex, of either type, has probability $ab\,w(v)/((a+b)V)$. If $a=b$, vertex-transitivity of the full group $A$ gives the same conclusion directly: the root probability is proportional to $w(v)$ over all vertices. Thus in both cases the untyped root law is the normalized weight $w$.

\textbf{Deleting the bad vertices.}
Call a vertex of $T_{a,b}$ good for $H$ if its $H$-stabilizer fixes all its neighbours. For either fixed type, badness is the Chabauty condition that $H$ meets $K\setminus K_1$, where $K$ is the vertex stabilizer in $A$ and $K_1$ fixes its radius-one ball pointwise. This is a compact open set, so the condition is clopen. The probability of badness tends to zero in the approximating IRSs. Invariance gives the same probability at every vertex of a given type. Since balls have uniformly bounded size for each fixed radius, the probability of any bad vertex in such a ball tends to zero.

Goodness descends to $X$. Adjacent good lifts have the same stabilizer: the stabilizer of either fixes the other, giving the two inclusions. Thus each connected component $D$ of the induced graph on good vertices has a constant positive weight $w$. Since the total weight is finite, $D$ is finite. Conditional on the root lying in $D$, it is uniform on $V(D)$.

Now choose $H$ with an approximating IRS law and form
$X=H\backslash T_{a,b}$, rooted at a vertex $v$ according to
\eqref{eq:tree-root}. A vertex of $X$ is good if its lifts are good;
this is well defined because goodness is $H$-invariant. For all
sufficiently large indices, condition on $v$ being good, that is,
restrict to outcomes with a good root and renormalize the probability
distribution. Delete all bad vertices of $X$ and their incident edges,
and keep the connected component $D$ containing $v$, with the same
root. The component is taken in the quotient graph $X$. By the
preceding paragraph it is finite, and in each fixed lattice quotient,
conditioning on the selected component makes the inherited root
uniform on that component. We obtain a countable convex combination
of uniformly rooted finite graphs of maximum degree at most
$\max\{a,b\}$. The conditioning changes asymptotically negligible
mass. Every quotient path lifts to the tree, so the rooted radius-$R$
neighbourhood is unaffected when the radius-$(R+1)$ tree ball consists
of good vertices.

Finally, the original quotient laws converge locally to $\mathcal L(\mu)$. A bounded quotient ball is determined by finitely many vertex and edge identifications in a bounded tree ball; these are tests for meeting compact open transporter sets. Each is a clopen Chabauty test. The retained finite-component laws therefore have the same limit. Truncating their countable convex combinations and approximating the coefficients by vertex proportions in disjoint unions produces a sequence of finite graphs with that limit.
\end{proof}

\begin{proposition}\label{prop:tree}
For every $a,b\geq3$ and every finitely generated non-sofic group $Q$,
there are a free uniform lattice $\Gamma<\Aut(T_{a,b})$ and
$N\triangleleft\Gamma$ with $\Gamma/N\cong Q$ such that
$\Ind_\Gamma^{\Aut(T_{a,b})}\delta_N$ is non-co-sofic.
\end{proposition}

\begin{proof}
Apply Proposition~\ref{prop:biregular-encoding} to $Q$, and put
$B=Q\backslash Y_Q$. The tower of coverings
\[
 T_{a,b}\longrightarrow Y_Q\longrightarrow B
\]
gives a free uniform lattice $\Gamma\cong\pi_1(B)<A^+$ and
$N\triangleleft\Gamma$ with quotient $Q$ and
$N\backslash T_{a,b}\cong Y_Q$. Every $\Gamma$-vertex stabilizer is
trivial. The covolume calculation in Lemma~\ref{lem:tree-extraction},
with the root choice in \eqref{eq:tree-root}, therefore gives equal
probability to the vertices of $B$. Hence the quotient law of the
induced IRS is $\nu_{Y_Q}$. If this IRS were co-sofic,
Lemma~\ref{lem:tree-extraction} and
Proposition~\ref{prop:biregular-encoding} would make $Q$ sofic.
\end{proof}

\endgroup

\begingroup
\section{Rank-one groups over non-Archimedean local fields}\label{sec:padic}

We first work with a unimodular locally compact group $G$ acting
continuously on $T_{a,b}$, with compact vertex stabilizers, such that
its type-preserving subgroup $G^+$ is transitive on edges.
The action is allowed to have compact kernel.

\begin{lemma}\label{lem:realization}
Under the hypotheses above, every finite connected bipartite $(a,b)$-biregular multigraph $B$ is $\Gamma\backslash T_{a,b}$ for a free uniform lattice $\Gamma<G^+$ acting freely on vertices.
\end{lemma}

\begin{proof}
Choose a spanning tree $D\subset B$, embed it in $T_{a,b}$ respecting types, and allocate the unused outward edges at each lifted vertex to the incident half-edges of $B\setminus D$. An embedding and these allocations exist because the degrees agree.

For an omitted edge from $v$ to $w$, write its two allocated outward edges as $(\widetilde v,x)$ and $(\widetilde w,y)$. Edge-transitivity of $G^+$ supplies $g_e$ with
\[
 g_e\widetilde w=x,\qquad g_ey=\widetilde v.
\]
The types match because $B$ is bipartite. Assigning $e\mapsto g_e$ defines a homomorphism from the free group $\pi_1(B)$ to $G^+$.

To verify faithfulness and the quotient, construct the universal covering tree of $B$ from copies of $D$ indexed by reduced words in the omitted edges. Map the copy with index $\gamma$ to $g_\gamma D$. The displayed equations make this a well-defined graph map, and the allocated half-edges make it locally bijective. A locally bijective map between connected trees is an isomorphism. It follows that the homomorphism is faithful, its image $\Gamma$ acts freely on vertices, and $\Gamma\backslash T_{a,b}=B$. This is the usual spanning-tree construction in Bass--Serre theory \cite{Serre80}.

A vertex stabilizer in $G$ is compact and open and meets $\Gamma$ trivially, so $\Gamma$ is discrete. Finiteness of $B$ and compactness of vertex stabilizers imply compactness of $G/\Gamma$. Since $G$ is unimodular, $\Gamma$ is a uniform lattice.
\end{proof}

\begin{theorem}\label{thm:tree-action}
Let $G$ satisfy the hypotheses above. For every finitely generated
non-sofic group $Q$, there are a free uniform lattice $\Gamma<G$
and $N\triangleleft\Gamma$ with $\Gamma/N\cong Q$ such that
$\mu_N^G:=\Ind_\Gamma^G\delta_N$ is non-co-sofic.
This IRS is discrete and ergodic and has compact conjugacy orbit.
\end{theorem}

\begin{proof}
Apply Proposition~\ref{prop:biregular-encoding} to the prescribed $Q$, and put $B=Q\backslash Y_Q$. Realize $B$ by Lemma~\ref{lem:realization}. The regular covering $Y_Q\to B$ determines $N\triangleleft\Gamma$ such that
\[
 \Gamma/N\cong Q,\qquad N\backslash T_{a,b}\cong Y_Q.
\]

Let $\pi:G\to H\leq A$ be the action homomorphism. Compact vertex stabilizers imply that $\pi$ is proper, with compact kernel and closed image. The restriction to $\Gamma$ is injective, by vertex-freeness. Write $\overline\Gamma=\pi(\Gamma)$ and $\overline N=\pi(N)$. The group $H$ is unimodular. It is cocompact in $A$: edge-transitivity gives $A=H A_{\{e\}}$ for a fixed edge $e$, whose setwise stabilizer is compact. Thus $A/H$ has an invariant probability measure.

If $\mu_N^G$ were co-sofic, Lemma~\ref{lem:induction}, first applied to $\pi$ and then to induction from $H$ to $A$, would show that
\[
 \mu_{\overline N}^A=\Ind_{\overline\Gamma}^A\delta_{\overline N}
\]
is co-sofic. Its subgroups are type-preserving and act freely on vertices. Moreover
\[
 \mathcal L(\mu_{\overline N}^A)=\nu_{Y_Q}.
\]
Indeed, all $\overline\Gamma$-vertex stabilizers are trivial. The covolume calculation in Lemma~\ref{lem:tree-extraction}, with \eqref{eq:tree-root}, therefore gives equal probability to the vertices of $B$, which correspond to the $Q$-orbits in $Y_Q$.

Lemma~\ref{lem:tree-extraction} would make $\nu_{Y_Q}$ graph-sofic, and Proposition~\ref{prop:biregular-encoding} would make $Q$ sofic, a contradiction. Finally, $N$ is discrete, and the map $G/\Gamma\to\Sub(G)$ is continuous and equivariant. Its image is a compact conjugacy orbit, and the push-forward of the transitive invariant probability measure is ergodic.
\end{proof}

\subsection{Proof of Theorem~\ref{thm:trees-padic}}
\begin{proof}
For $\Aut(T_{a,b})$, the assertion is Proposition~\ref{prop:tree}.
Now let $k$ be a non-Archimedean local field and $\mathbf G$ a
connected simple $k$-group of $k$-rank one.
Bruhat--Tits theory gives an action of $\mathbf G(k)$ on a locally
finite thick tree $T_{a,b}$, with $a,b\geq3$. Vertex stabilizers are
compact and open, the type-preserving subgroup is edge-transitive,
and $\mathbf G(k)$ is unimodular. These facts hold in both
characteristics; see \cite[\S\S1.1, 1.3]{Baum03} and
\cite[\S3.2.3]{RTW15}. Apply Theorem~\ref{thm:tree-action}.
\end{proof}

The proof uses neither an analytic inverse function theorem nor an
exclusion of small torsion elements. Approximating lattices may have
nontrivial finite vertex stabilizers; Lemma~\ref{lem:tree-extraction}
removes their effect locally. Compact kernels include finite centers
and compact anisotropic factors. The side-pairing generators still
generate a free lattice in the original group, because their action
on the covering tree is faithful. No surjectivity assertion for
central isogenies on $k$-points is needed.

\endgroup

\section{Consequences and scope}

\begin{corollary}
For every $n\geq2$, there is a regular covering $X$ of a
closed oriented hyperbolic $n$-manifold such that the law of
$X$ {with a basepoint chosen uniformly} in a fundamental domain of its deck
group is not a Benjamini--Schramm limit of finite-volume
oriented hyperbolic orbifolds. The injectivity radius of
$X$ has a positive uniform lower bound.
{For every $a,b\geq3$, there is an analogous regular covering
of a finite bipartite $(a,b)$-biregular graph whose uniform-orbit
rooted law is not a local weak limit of finite graphs.}
\end{corollary}

\begin{proof}
For the manifold assertion take
$X=N\backslash\HH^n$ from {Remark~\ref{rem:prescribed-quotient}},
working in $\operatorname{SO}^{+}(n,1)$. Its injectivity
radius is bounded below by that of the closed base manifold.
{Pointed isometries are required to preserve orientation.
Choose an oriented orthonormal frame uniformly at the basepoint.
Averaging over the compact frame group $\operatorname{SO}(n)$ lifts
convergence of the pointed laws to convergence of the framed laws.
The IRS--Benjamini--Schramm correspondence
\cite[Corollary~3.11 and the subsequent orbifold extension]{ABBGNRS17}
(see also \cite{AB22,Gel18}) then turns the asserted approximation into
co-soficity of $\Ind_\Gamma^G\delta_N$, a contradiction.\par}
{For graphs use $Y_Q$ from
Proposition~\ref{prop:biregular-encoding} for a non-sofic group $Q$.
A hypothetical approximating sequence can be assumed to have maximum
degree at most $\max\{a,b\}$: delete vertices of larger degree and
their incident edges. This preserves the local limit, because the
probability that any fixed rooted ball meets such a vertex tends to zero.}
\end{proof}

\begin{remark}
{Proposition~\ref{thm:lie-transfer}} applies to any normal
subgroup of a cocompact lattice in a real Lie group whose
centralizer is discrete and whose quotient is finitely
generated. Thus in the real cases any such non-sofic
quotient gives the required IRS. The tree argument uses
the particular graph encoding of Section~\ref{sec:graphs}.
\end{remark}

\begin{remark}
No torsion-free or cocompactness restriction is imposed on
the approximating lattices. In the Lie-group argument,
finite invariant volume supplies finite sets of points on
a local transversal. In the tree argument, the covolume
weights and finite good components in
Lemma~\ref{lem:tree-extraction} allow arbitrary tree lattices.
\end{remark}

We conclude with a general conjecture suggested by these constructions.
For a countable group, soficity means that every finitely generated
subgroup is sofic.

\begin{conjecture}\label{conj:induction}
Let $G$ be a locally compact second countable group, let
$\Gamma<G$ be a uniform lattice, and let $N\triangleleft\Gamma$.
If $\Gamma/N$ is non-sofic, then
$\Ind_\Gamma^G\delta_N$ is a non-co-sofic IRS of $G$.
\end{conjecture}

{Proposition~\ref{thm:lie-transfer}} proves the conjecture when $G$
is a real Lie group, $C_G(N)$ is discrete, and $\Gamma/N$ is
finitely generated. In its full generality, the conjecture would
{give Theorems~\ref{thm:main} and~\ref{thm:trees-padic} from the existence of non-sofic
quotients of uniform lattices in the groups under consideration.}

The conjecture would also imply that if $G$ admits one residually
finite uniform lattice $\Delta$, then every uniform lattice in $G$
is sofic. Indeed, choose finite-index normal subgroups
$\Delta_j\triangleleft\Delta$ decreasing to $\{1\}$. Then
\[
 \mu_{\Delta_j}^{G}
 =\int_{G/\Delta}\delta_{g\Delta_jg^{-1}}\dd m(g\Delta)
 \longrightarrow\delta_{\{1\}}.
\]
Thus $\delta_{\{1\}}$ is co-sofic, and the conjecture applies with
$N=\{1\}$. In particular, taking the residually finite lattice
$F_r\times F_s$ in
$\Aut(T_{2r})\times\Aut(T_{2s})$, for $r,s\geq2$, would establish
soficity of every uniform lattice in this product, including the
Burger--Mozes examples.\footnote{While we prove the conjecture for certain lattices in
$\Aut(T)$, we do not know how to deal with lattices in
$\Aut(T_{2r})\times\Aut(T_{2s})$.} See \cite[Question~2 and the subsequent
discussion]{Timar12} for the related problem of recovering group
soficity from finite approximations to an unlabelled Cayley graph.

\section*{Acknowledgments}
The author was partially supported by the National Science Foundation
under grant DMS-2505196. The author acknowledges the assistance of
OpenAI's GPT-6 Astra in developing the arguments and preparing the
manuscript.

\enlargethispage{2\baselineskip}

\end{document}